\documentclass[10pt]{article}
\usepackage{amsfonts,amssymb,amsmath,comment}
\usepackage{graphicx,graphics,color}
\usepackage[english]{babel}

 \makeatletter
 \@addtoreset{equation}{section}
 \makeatother

 \newcounter{enunciato}[section]

 \newtheorem{ittheorem}{Theorem}
 \newtheorem{itlemma}{Lemma}
 \newtheorem{itproposition}{Proposition}
 \newtheorem{itdefinition}{Definition}
 \newtheorem{itcorollary}{Corollary}
 \newtheorem{itconjecture}{Conjecture}

 \newenvironment{theorem}{\addtocounter{enunciato}{1}
 \begin{ittheorem}}{\end{ittheorem}}

 \newenvironment{lemma}{\addtocounter{enunciato}{1}
 \begin{itlemma}}{\end{itlemma}}

\newcommand{\halmos}{\rule{1ex}{1.4ex}}

\newenvironment{proof}{\noindent {\em Proof}.\,\,}
   {\hspace*{\fill}$\halmos$\medskip}

\def \qed {{\hspace*{\fill}$\halmos$\medskip}}

\def \ba {\begin{array}}
\def \ea {\end{array}}

\def \Z {{\mathbb Z}}
\def \R {{\mathbb R}}

\def \N {{\mathbb N}}

\begin{document}
\title{Maximising Neumann eigenvalues on rectangles}

\author{\renewcommand{\thefootnote}{\arabic{footnote}}
M.\ van den Berg
\footnotemark[1]
\\
\renewcommand{\thefootnote}{\arabic{footnote}}
D.\ Bucur
\footnotemark[2]
\\
\renewcommand{\thefootnote}{\arabic{footnote}}
K.\ Gittins
\footnotemark[3]
}

\footnotetext[1]{School of Mathematics, University of Bristol, University Walk, Bristol BS8 1TW, United Kingdom.\, {\tt mamvdb@bristol.ac.uk} }

\footnotetext[2]{ Institut Universitaire de France and Laboratoire de Math\'{e}matiques, CNRS UMR 5127, Universit\'{e} de Savoie Campus Scientifique, 73376 Le-Bourget-du-Lac, France.\, {\tt dorin.bucur@univ-savoie.fr}}

\footnotetext[3]{School of Mathematics, University of Bristol, University Walk, Bristol BS8 1TW, United Kingdom.\, {\tt kg13951@bristol.ac.uk}}

\date{{ 29 July 2016}}

\maketitle

\begin{abstract}
We obtain results for the spectral optimisation of Neumann eigenvalues on rectangles in $\R^2$ with a measure or perimeter constraint. We show that the rectangle with measure $1$ which maximises the $k$'th
Neumann eigenvalue converges to the unit square in the Hausdorff metric as $k\rightarrow \infty$. Furthermore, we determine the unique maximiser of the $k$'th Neumann eigenvalue on a rectangle with given perimeter.

\vskip 0.5truecm
\noindent
{\it AMS} 2000 {\it subject classifications.} 35J20, 35P99.

\medskip\noindent
{\it Acknowledgements.} MvdB acknowledges support by The Leverhulme Trust
through International Network Grant \emph{Laplacians, Random Walks, Bose Gas,
Quantum Spin Systems}.  DB is a member of the ANR {\it Optiform}  programme ANR-12-BS01-0007. KG was supported by an EPSRC DTA. The authors thank
Beniamin Bogosel for very helpful numerical assistance and the referee for
{ helpful} suggestions.

\end{abstract}

\section{Introduction}\label{Introduction}
Let $\Omega$ be an open or quasi-open set in Euclidean space $\R^m \; (
m=2,3,\dots)$, with boundary $\partial \Omega$, and let
$-\Delta_{\Omega}$ be the Dirichlet Laplacian acting in
$L^2(\Omega)$. It is well known that if $\Omega$ has finite
Lebesgue measure $|\Omega|$ then $-\Delta_{\Omega}$ has compact
resolvent, and the spectrum of $-\Delta_{\Omega}$ is discrete and
consists of eigenvalues $\lambda_1(\Omega)\le\lambda_2(\Omega)\le
\dots$ with $\lambda_j(\Omega)\rightarrow\infty$ as
$j\rightarrow\infty$.
The problem of minimising the eigenvalues of the Dirichlet Laplacian over sets in $\R^m$ with a geometric constraint has been studied extensively.
For example it was shown in \cite{B1} and \cite{MP} that for any $k\in \N$ the minimisation problem
\begin{equation}\label{e2a}\inf \{\lambda_k(\Omega) :\Omega\
\textup{{quasi-}open in}\ \R^m ,\ |\Omega| = c \}
 \end{equation}
has a bounded minimiser with finite perimeter. The celebrated Faber--Krahn and Krahn--Szeg\"o inequalities assert that these minimisers are a ball with measure $c$ for $k=1$ and the union of two disjoint balls each with measure $c/2$ for $k=2$ respectively, see \cite{H}. It has been conjectured that if $m=2,k=3$ the disc with measure $c$ is a minimiser.
Less is known for higher values of $k$.
For $m=2, k \geq 5$, it was shown in \cite{Be} that neither the disc nor a disjoint union of discs is optimal.
In addition, numerical experiments { indicate} as to what the minimisers look like see \cite{Ou, AF1}. Some bounds on the number of components of minimisers of \eqref{e2a} have been obtained in \cite{vdBI}.

Other constraints than the measure have been considered in \cite{BBH}, \cite{DPV}, \cite{BF} and \cite{vdB}. For example, it was shown in \cite{DPV} that a minimiser exists for the $k$th Dirichlet eigenvalue under the constraint that the perimeter is fixed and the measure is finite. Existence in the planar case is particularly straightforward, since elements of minimising sequences are convex and bounded uniformly in $k$. The latter fact allowed Bucur and Freitas to show in \cite{BF} that there exists a sequence of translates of these minimisers that converges to the disc in the Hausdorff metric. This phenomenon of an asymptotic shape has been established for a wide class of constraints in \cite{vdB}. However, this class does not include the original measure constraint.

Numerical experiments have also been carried out to investigate the optimisation of Dirichlet eigenvalues
subject to a perimeter constraint, see \cite{AF2} and \cite{BoOu}.
These papers use different methods to obtain insight as to what the optimal shapes would look like.
 The asymptotic behaviour of the $k$th optimal eigenvalue on $m$-dimensional cuboids (rectangular parallelepipeds) with a perimeter constraint was analysed in \cite{AF2}.

In \cite{CoSo} it was shown that the infimum in \eqref{e2a} with $c=1$ behaves like $4\pi ^2k^\frac 2m\omega_m^{-\frac 2m}$ as $k \rightarrow \infty$ provided the P\' olya conjecture for Dirichlet eigenvalues holds. That is for every bounded open set $\Omega \subset \R^m$, $\lambda_k(\Omega)\ge 4\pi^2 (|\Omega|\omega_m)^{-\frac 2m}k^\frac 2m$, where $\omega_m$ is the measure of the ball in $\R^m$ with radius $1$.

In a recent paper, \cite{AF}, Antunes and Freitas proved the following asymptotic shape result with a measure constraint.
For $a\ge 1$, let
\begin{equation*}
R_a=\{(x_1,x_2):0<x_1<a,\, 0<x_2<a^{-1}\}
\end{equation*} be a rectangle with measure $1$. The infimum of the variational problem

\begin{equation*}
\lambda_k^*:=\inf\{\lambda_k(R_a)\}
\end{equation*}
is achieved for some $a_k^*\ge 1$, and $\lim_{k\rightarrow \infty}a_k^*=1$.

A heuristic explanation for this asymptotic shape result is the following (see \cite{AF}). For any rectangle in $\R^2$ with measure $|R|$ and perimeter $\textup{Per}(R)$
one has that
\begin{equation}\label{e2c}
\lambda_k(R)=\frac{4\pi k}{|R|}+\frac{2\pi^{1/2}\textup{Per}(R)k^{1/2}}{|R|^{3/2}}+o(k^{1/2}),\ k \rightarrow \infty.
\end{equation}
So if $|R|=1$ then \eqref{e2c} suggests that the rectangle that minimises $\lambda_k(R),\ k \rightarrow \infty$ is the one with minimal perimeter, i.e. the unit square.
The main part of the proof in \cite{AF} is to show that the $a_k^*$'s are uniformly bounded. It is then possible to use well-known number theoretic results for the number of lattice points
inside ellipses where the ratio of the axes remains bounded.

The asymptotic formula \eqref{e2c} holds true for a wide class of planar domains with a smooth boundary that satisfy a billiard condition. This suggests that the asymptotic shape with fixed measure is a disc. The proof of this seems well beyond reach, even if an additional convexity constraint is imposed, \cite{vdB}.

In this paper we consider the maximisation of Neumann eigenvalues. It is well known that if $\Omega$ is an open, bounded  and connected set in $\R^m$ with Lipschitz boundary then the spectrum of the Neumann Laplacian is discrete and consists of eigenvalues   $\mu_0(\Omega)<\mu_1(\Omega)\le \mu_2(\Omega)\le \dots$ accumulating at infinity. The first Neumann eigenvalue has multiplicity $1$ and $\mu_0(\Omega)=0$. Szeg\"o and Weinberger showed that $\mu_1(\Omega)\le \mu_1(\Omega^*)$, where $\Omega^*$ is the ball with the same measure as $\Omega$, {see \cite{H}}. It was shown in \cite{GNP} that { the union of} two disjoint planar discs, each with measure { $c/2$,} achieves the supremum of $\mu_2(\Omega)$ in the class of simply connected sets in $\R^2$ {with measure $c$}. Nothing is known about the existence of maximisers for higher $k$ (see, for instance, \cite[Subsection 7.4]{DB}).
In this paper, we consider the problem of maximising the $k$'th Neumann eigenvalue over all rectangles in $\R^2$ with fixed measure, and study the asymptotic behaviour as $k \rightarrow \infty$.

Our main result is the following.
\begin{theorem}\label{the1}
\begin{enumerate}
\item[\textup{(i)}] Let $k\in \N$. The variational problem
 \begin{equation}\label{i1}
\mu_k^*:=\sup\{\mu_k(R_b) :  b\ge 1\}
\end{equation}
has a maximising rectangle $R_b$ with $b=b_k^*$.
\item[\textup{(ii)}] Any sequence of optimal rectangles $(R_{b_k^*})$ converges in the Hausdorff metric to the unit square as $k\rightarrow\infty$. Moreover
there exists $\theta\in (\frac12,1)$ such that for $k\rightarrow \infty$,
\begin{equation}\label{i4}
b_k^*=1+O(k^{(\theta-1)/4}).
\end{equation}
\item[\textup{(iii)}] Let $\mu_k^*=\mu_k(R_{b_k^*})$.
Then
\begin{equation}\label{i3}
\mu_k^*=4\pi k-8(\pi k)^{1/2}+O(k^{(\theta +1)/4}),\, k\rightarrow \infty.
\end{equation}
\end{enumerate}
\end{theorem}
The exponent $\theta$ shows up in the remainder of Gauss' circle problem.
It is known that for any $\epsilon>0$, (see the Introduction in \cite{Hux})
\begin{equation*}
\theta=\frac{131}{208}+\epsilon.
\end{equation*}

The table below shows that the maximising rectangles for $k=4,6,10$ and $k=15$ are not unique.
The eigenvalues of the rectangle $R_b$ are of the form
\begin{equation}\label{ii}
\mu_{p,q}= \frac{\pi^2 p^2}{b^2}+ \pi^2 q^2 b^2,
\end{equation}
for $p,q \in \mathbb Z^+=\N\cup\{0\}$. The ordered list of real numbers $\{\mu_{p,q}:\, p\in \Z^+, q\in \Z^+\}$
are the eigenvalues $\{0=\mu_0(R_b)<\mu_1(R_b)\le \mu_2(R_b)\le \dots \}$ of the Neumann Laplacian on $R_b$.
From the proof of Theorem \ref{the1}(ii) we will see that the maximised $k$th eigenvalue has multiplicity at least $2$. In the table below we list the values of $\mu_k^*$ for $k=1,\dots,15$ as well as the $b_k^*$ and the pairs of maximising modes that realise this maximum.

\begin{center}
\begin{tabular}{|c | c | c | c |}
\hline
$k$ & $\mu_k^*$ & $b_k^*$ & \textup{Maximising pair of modes}\\
\hline
$1$ & $\pi^2$                     & $1$ & $(1,0),(0,1)$\\
\hline
$2$ & $2\pi^2$                    & $\sqrt 2$ & $(2,0),(0,1)$\\
\hline
$3$ & $3\pi^2$                    & $\sqrt 3$ & $(3,0),(0,1)$\\
\hline
$4$ & $4\pi^2$                    & $2$\, \textup{or}\, $1$ & $(4,0),(0,1)$ \textup{or}\, $(2,0),(0,2)$\\
\hline
$5$ & $5\pi^2$                    & $\sqrt 5$ & $(5,0),(0,1)$\\
\hline
$6$ & $6\pi^2$                    & $\sqrt 6$ \textup{or}\, $\frac12\sqrt 6$&$(6,0),(0,1)$ \textup{or}\, $(3,0),(0,2)$\\
\hline
$7$ & $7\pi^2$                    & $\sqrt 7$ & $(7,0),(0,1)$\\
\hline
$8$ & $\frac{18\sqrt 5}{5}\pi^2$   & $\frac {\sqrt 2}{2} 5^{1/4}$ & $(2,2),(3,0)$\\
\hline
$9$ & $\frac{16\sqrt 3}{3}\pi^2$  & $\frac{2}{3^{1/4}}$& $(4,1),(0,2)$\\
\hline
$10$&$10\pi^2$                    & $\frac12 \sqrt {10}$ \textup{or} $\sqrt {10}$ & $(5,0),(0,2)$ \textup{or}\, $(10,0),(0,1)$\\
\hline
$11$ &$12\pi^2$                    & $\frac23\sqrt 3$ & $(4,0),(0,3)$\\
\hline
$12$ &$\frac{77}{20}\sqrt{10}\pi^2$& $\left(\frac{8}{5}\right)^{1/4}$& $(1,3),(3,2)$\\
\hline
$13$ & $8\sqrt 3\pi^2$& $ 3^{1/4}\sqrt 2$                           &$(6,1),(0,2)$\\
\hline
$14$   & $15\pi^2$ & $\frac13 \sqrt {15}$                                         &$(5,0),(0,3)$\\
\hline
$15$   & $16\pi^2$ & $2$\, \textup{or} $1$ &   $(8,0),(0,2)$ \textup{or}\, $(4,0),(0,4)$\\
\hline
\end{tabular}
\end{center}

We also see in the table above that the unit square is a maximiser for $k=1,4$ and $k=15$. We conjecture that the unit square is a maximiser if the maximising pair of modes are given by $(2^n,0),(0,2^n):n\in \Z^+$. This gives that the unit square is a maximiser for $\mu_k$ if
\begin{equation*}
k=\sum_{l\in \Z^+}\lfloor (4^n-l^2)_+^{1/2} \rfloor +2^n-1, \, n\in \Z^+.
\end{equation*}

The heuristic explanation of \eqref{i4} is that for Neumann eigenvalues on a rectangle $R\subset \R^2$,
\begin{equation*}
\mu_k(R)=\frac{4\pi k}{|R|}-\frac{2\pi^{1/2}\textup{Per}(R)k^{1/2}}{|R|^{3/2}}+o(k^{1/2}),\ k \rightarrow \infty,
\end{equation*}
so that the maximising rectangle with measure $|R|$ is the one that minimises its perimeter, i.e. the square with measure $|R|$.

The key ingredient in the proof of \eqref{i4} in Section \ref{The1} below is to show that $\limsup_{k\rightarrow \infty} b_k^*<\infty.$ This is more involved than the corresponding proof of Antunes and Freitas that $\limsup_{k\rightarrow \infty} a_k^*<\infty$ for the minimising rectangles of the Dirichlet eigenvalues. In particular, it requires an a priori bound on $\limsup b_k^*/k^{1/2}$ with some constant which, for technical reasons, has to be sufficiently small. This is achieved in Lemma \ref{lem3}. The number theoretical estimates are also more involved, and will be given in Lemma \ref{lem1}.

In Section~\ref{Per}, we turn our attention to the optimisation of Neumann eigenvalues on rectangles with a perimeter constraint. Generally, these problems are not well-posed (see Section~\ref{Per} for a discussion). Thus, we consider
the following variational problems
\begin{equation}\label{bbg04}
\sup\{\mu_k (R) : R \mbox{ rectangle,} \; \textup{Per} (R)=4\},
\end{equation}
and
\begin{equation}\label{bbg05}
\inf\{\mu_k (R) : R \mbox{ rectangle,} \; \textup{Per} (R)=4\}.
\end{equation}
In Subsection~\ref{S3.1}, we consider problem \eqref{bbg04} and we prove that for $k \in \N$,
there is a unique maximising rectangle for $\mu_k$ that collapses to a segment as $k \rightarrow \infty$.
In Subsection~\ref{S3.2}, we show that for $k=1$ problem \eqref{bbg05} does not have a solution,
while for $k \geq 2$ it does and any sequence of minimising rectangles converges to the unit square in the sense of
Hausdorff as $k \rightarrow \infty$.

\section{Proof of Theorem \ref{the1}}\label{The1}
\noindent{\it Proof of Theorem \ref{the1}(i).} Fix $k \in \N$. Suppose that $\{R_{b^{(\ell)}}\}_{\ell \in \N}$ is a maximising sequence for
$\mu_{k}$ such that $b^{(\ell)} \rightarrow \infty$ as $\ell \rightarrow \infty$. Then, for sufficiently large $\ell$,
\begin{equation*}
\mu_{k}(R_{b^{(\ell)}}) = \frac{\pi^{2}k^{2}}{(b^{(\ell)})^{2}},
\end{equation*}
and so $\mu_{k}(R_{b^{(\ell)}}) \rightarrow 0$ as $\ell \rightarrow \infty$. On the other hand, we have that $b=1$
for a square and so $\mu_{k}\geq \pi^{2} >0$. This contradicts the assumption
that $\{R_{b^{(\ell)}}\}_{\ell \in \N}$ is a maximising sequence for $\mu_{k}$. Thus any maximising
sequence $\{R_{b^{(\ell)}}\}_{\ell \in \N}$ for $\mu_{k}$ is such that $b^{(\ell)}$ remains bounded. Hence there exists a convergent subsequence, again denoted by
$b^{(\ell)}$, such that $b^{(\ell)}\rightarrow b_k^*$ for some $b_k^*\in[1,\infty)$. Since $b\mapsto \mu_{k}(R_b)$ is
continuous, $\mu_k(R_{b^{(\ell)}})\rightarrow \mu_k(R_{b_k^*})$ as $\ell \rightarrow \infty$. Hence $R_{b_k^*}$ is a maximiser.

In order to prove Theorem \ref{the1}(ii), we need three lemmas that will be given below.
\begin{lemma}\label{lem2} Let
$\nu_k=\mu_k(R_1)$ be the $k$'th positive
Neumann eigenvalue for the unit square in $\R^2$. Then
\begin{equation}\label{e17}
\nu_k \ge 4 \pi k -16(\pi k)^{1/2},\, k  \ge 1.
\end{equation}
\end{lemma}
\begin{proof}
The cases $k=1,2,\dots, 5$ hold true by direct computation. Let us assume that $k\ge 6$.
For the unit square we have by \eqref{e3} that
\begin{equation*}
N(\nu;1)=2\bigg\lfloor \frac{\nu^{1/2}}{\pi} \bigg\rfloor+\big\vert\left\{ (x,y) \in \N^{2} : x^2 + y^{2} \leq \nu/\pi^2 \right\}\big\vert.
\end{equation*}
Let $\nu >2$. For each lattice point in $\N^2$ (i.e. $x \ge 1, y\ge 1$) satisfying $x^2+y^2\le \nu/\pi^2$ there exists an open lower left-hand square with vertices $(x,y),(x-1,y),(x-1,y-1),(x,y-1)$
inside the quarter circle with radius $\nu^{1/2}/\pi$ in the first quadrant. Hence
\begin{equation*}
N(\nu;1)\le\frac{\nu}{4\pi}+\frac{2\nu^{1/2}}{\pi}.
\end{equation*}
So for $\nu=\nu_k$ we have that
\begin{equation}\label{e20}
k\le\frac{\nu_k}{4\pi}+\frac{2\nu_k^{1/2}}{\pi}.
\end{equation}
We note that \eqref{e20} also holds in case $\nu_k$ has multiplicity larger than $1$.
Since the unit square tiles $\R^2$, we have by P\'olya's Inequality, \cite{P},
that $\nu_k\le 4\pi k$. Hence
\begin{equation*}
k\le\frac{\nu_k}{4\pi}+\frac{2(4\pi k)^{1/2}}{\pi}.
\end{equation*}
This implies \eqref{e17}.
\end{proof}

\begin{lemma}\label{lem1} For all $\mu>0,\, b>0$, define the counting function
\begin{equation}\label{e3}
N(\mu;b) = \bigg\vert\left\{ (x,y) \in ({ \Z^{+}})^{2} \setminus \{(0,0)\} : \frac{\pi^2x^{2}}{b^{2}} +\pi^{2}b^{2}y^{2} \leq \mu \right\}\bigg\vert.
\end{equation}
Then for all $\mu>0,\, b>0$ with $\frac{\mu^{1/2}}{b\pi}\ge2$ we have that
\begin{equation}\label{e2}
N(\mu;b) \ge \frac{\mu}{4\pi} + \frac{b\mu^{1/2}}{2\pi} - \frac{b^{3/2}\mu^{1/4}}{(2\pi)^{1/2}} -1.
\end{equation}
\end{lemma}

To prove Lemma~\ref{lem1}, we obtain a lower bound for the number of integer lattice points in $\N^2$ that are inside or on the ellipse
\begin{equation*}
E(\mu) = \left\{(x,y) \in \R^{2} : \frac{\pi^2x^{2}}{b^{2}} + \pi^{2}b^{2}y^{2} \leq \mu \right\}.
\end{equation*}
\begin{proof}
For each $(x,y) \in E(\mu)$, we have that
\begin{equation*}
x \leq \frac{b}{\pi}(\mu - \pi^{2}b^{2}y^{2})_{+}^{1/2} = b^{2}\left(\frac{\mu}{\pi^{2}b^{2}} - y^{2}\right)_{+}^{1/2}.
\end{equation*}
Then
\begin{align*}
N(\mu;b) &= \left \lfloor{\frac{\mu^{1/2}}{\pi b}}\right \rfloor + \left \lfloor{\frac{b\mu^{1/2}}{\pi}}\right \rfloor
+\sum_{y \in \mathbb{N}} \left \lfloor{b^{2}\left(\frac{\mu}{\pi^2b^{2}} - y^{2}\right)_{+}^{1/2}}\right \rfloor \notag\\
&= \left \lfloor{\frac{\mu^{1/2}}{\pi b}}\right \rfloor + \left \lfloor{\frac{b\mu^{1/2}}{\pi}}\right \rfloor
+\sum_{y=1}^{\left \lfloor{\mu^{1/2}/(\pi b)}\right \rfloor} \left \lfloor{b^{2}\left(\frac{\mu}{\pi^2b^{2}} - y^{2}\right)^{1/2}}\right \rfloor \notag\\
&\geq \left \lfloor{\frac{\mu^{1/2}}{\pi b}}\right \rfloor + \left \lfloor{\frac{b\mu^{1/2}}{\pi}}\right \rfloor
+\sum_{y=1}^{\left \lfloor{\mu^{1/2}/(\pi b)}\right \rfloor} b^{2}\left(\frac{\mu}{\pi^2b^{2}} - y^{2}\right)^{1/2}
-\left \lfloor{\frac{\mu^{1/2}}{\pi b}}\right \rfloor \notag\\
&= \left \lfloor{\frac{b\mu^{1/2}}{\pi}}\right \rfloor
+b^{2}\sum_{y=1}^{\left \lfloor{\mu^{1/2}/(\pi b)}\right \rfloor} \left(\frac{\mu}{\pi^2b^{2}} - y^{2}\right)^{1/2}.
\end{align*}
Let $R=\frac{\mu^{1/2}}{\pi b}$ and define $f(y):=(R^{2} - y^{2})^{1/2},\, 0\le y \le R$.
Then $\sum_{y=1}^{\lfloor{R} \rfloor} (R^{2} - y^{2})^{1/2}$ is the area
of the rectangles that are inscribed in the first quadrant of the circle of radius $R$. Hence, we can rewrite this as
\begin{equation}\label{e6}
\sum_{y=1}^{\left \lfloor{R}\right \rfloor} (R^{2} - y^{2})^{1/2} = \frac{\pi R^{2}}{4} - A,
\end{equation}
where
\begin{align}\label{e7}
A=\sum_{n=0}^{\lfloor R \rfloor-1}\int_{n}^{n+1}(f(y)-f(n+1))dy+\int _{\lfloor R \rfloor}^R f(y)dy.
\end{align}
Since $\lfloor R \rfloor\le y \le R$ we have that
\begin{align}\label{e8}
\int _{\lfloor R \rfloor}^R f(y)dy&= \int_{\lfloor R\rfloor}^R(R-y)^{1/2}(R+y)^{1/2}\, dy\nonumber \\ &\le (2R)^{1/2}\int_{\lfloor R\rfloor}^R(R-y)^{1/2}\, dy
\nonumber \\ &= \frac23 (R-\lfloor R \rfloor)^{3/2} (2R)^{1/2}.
\end{align}
Since $f$ is decreasing and concave, we have that
\begin{equation*}
f(y)\le f(n)+ (y-n)f'(n),\, n \le y \le n+1.
\end{equation*}
Hence
\begin{equation}\label{e10}
\sum_{n=0}^{\lfloor R \rfloor-1}\int_{n}^{n+1}(f(y)-f(n+1))dy\le f(0)-f(\lfloor R\rfloor)+\frac12\sum_{n=0}^{\lfloor R \rfloor-1}f'(n).
\end{equation}
Since
\begin{equation*}
f'(y)= -\frac{y}{(R^{2}-y^{2})^{1/2}},
\end{equation*}
$f'(0)=0$, and $y\mapsto -f'(y)$ is increasing, we have that
\begin{equation}\label{e12}
\frac12\sum_{n=0}^{\lfloor R \rfloor-1}f'(n)\le \frac12 \int_0^{\lfloor R\rfloor -1}f'(y)dy=\frac12(f(\lfloor R\rfloor -1)-f(0)).
\end{equation}
By \eqref{e6}-\eqref{e8}, \eqref{e10}, and \eqref{e12}
\begin{align*}
\sum_{y=1}^{\left \lfloor{R}\right \rfloor} (R^{2} - y^{2})^{1/2} &\ge \frac{\pi R^{2}}{4}-\frac{R}{2}+(R^2-\lfloor R\rfloor^2)^{1/2}-\frac12(R^2-(\lfloor R\rfloor-1)^2)^{1/2}
-\frac23 (R-\lfloor R \rfloor)^{3/2} (2R)^{1/2}.
\end{align*}
Next note that
\begin{equation*}
\frac12(R^2-(\lfloor R\rfloor-1)^2)^{1/2}\le \frac12(R-\lfloor R\rfloor +1)^{1/2}(2R)^{1/2},
\end{equation*}
and that for $R\ge 2$,
\begin{equation*}
(R^2-\lfloor R\rfloor^2)^{1/2}\ge (R-\lfloor R \rfloor)^{1/2}\bigg(\frac{R+\lfloor R\rfloor}{2R}\bigg)^{1/2}(2R)^{1/2}\ge \left(\frac{5}{6}\right)^{1/2}(R-\lfloor R \rfloor)^{1/2}(2R)^{1/2}.
\end{equation*}
Let $\beta=R-\lfloor R\rfloor\in [0,1]$ and define $ g:[0,1]\mapsto \R$ by
\begin{equation*}
g(\beta)=-\bigg(\frac{5}{6}\bigg)^{1/2}\beta^{1/2}+\frac12(\beta+1)^{1/2}+\frac23\beta^{3/2}.
\end{equation*}
Then
\begin{equation*}
g''(\beta)=\frac14\bigg(\frac{5}{6}\bigg)^{1/2}\beta^{-3/2}-\frac18(\beta+1)^{-3/2}+\frac12\beta^{-1/2}>0,
\end{equation*}
and so $g$ is convex. Hence $g(\beta)\le \max\{g(0), g(1)\}=\frac12.$
So for $R=\frac{\mu^{1/2}}{\pi b}\ge2$ we have that
\begin{align}\label{e14}
N(\mu;b) &\geq \left \lfloor{\frac{b\mu^{1/2}}{\pi}}\right \rfloor + b^{2}\left(\frac{\pi R^{2}}{4}
-\frac{R}{2} -  \frac{R^{1/2}}{2^{1/2}}\right) \notag\\
&= \left \lfloor{\frac{b\mu^{1/2}}{\pi}}\right \rfloor + \frac{\mu}{4\pi} - \frac{b\mu^{1/2}}{2\pi}
- \frac{b^{3/2}\mu^{1/4}}{(2\pi)^{1/2}}  \notag\\
&\geq \frac{\mu}{4\pi} + \frac{b\mu^{1/2}}{2\pi} - \frac{b^{3/2}\mu^{1/4}}{(2\pi)^{1/2}} -1.
\end{align}
\end{proof}

 Below we obtain an a priori upper bound on  the longest side $b_k^*$ of a maximising rectangle in terms of $k$.

\begin{lemma}\label{lem3} We have that
\begin{equation}\label{e23}
\limsup_{{k \rightarrow \infty}} \frac{b_k^*}{k^{1/2}} \le 0.46359.
\end{equation}
\end{lemma}
\begin{proof}
Define $c_k:=\frac{b_k^*}{k^{1/2}}$. We shall bound $c_k$ using the maximality of $\mu_k^*$ at $R_{b_k^*}$.
We first note that
\begin{equation}\label{e25}
\limsup_{{k \rightarrow \infty}} c_k \le \bigg(\frac{\pi}{4}\bigg)^{1/2}.
\end{equation}
Indeed, we know by \eqref{ii} that the eigenvalues of $R_{b_k^*}$ are of the form
$$ \mu_{p,q}= \frac{\pi^2 p^2}{c_k^2 k}+ \pi^2 q^2 c_k^2 k,$$
for $p,q \in \mathbb Z^+$.
 Choosing the pairs $(p,q)$ in \eqref{ii} as $(0,0), (1,0),\dots, (k,0)$  we get that
$$\frac {\pi^2 k^2}{c_k^2 k} \ge \mu_k^* \ge \nu_k.$$
This gives by Lemma \ref{lem2} that for $k\ge 6$,
$$c_k^2 \le \frac{\pi^2 k}{4 \pi k -16(\pi k)^{1/2}},$$
which passing to the limit leads to \eqref{e25}.

Assume now that for some $k$ (large), all of the eigenvalues of $R_{b_k^*}$ up to index $k$ are given by the pairs $(p,q) = (0,0), (1,0),\dots, (k,0)$. If this is the case, then we see
that $\mu_k^*$ has to be (at least) double, and hence equal to some value of the form
$$\frac{\pi^2 p^2}{c_k^2 k}+ \pi^2 q^2 c_k^2 k$$
for some $q\ge 1$. Indeed, if it is not double, then being simple, for a small variation of $b$ around $b_k^*$ it continues to be simple and we can perform the derivative of the mapping
$$b \mapsto \mu_k(R_b),$$
in $b_k^*$. This derivative equals $- \frac{2 \pi ^2 k^2}{(b_k^*)^3}$ which is not vanishing, in contradiction with the optimality of $b_k^*$.

So, either the first $k+1$ eigenvalues are not given by $(p,q) = (0,0), (1,0),\dots, (k,0)$, or the value of $\mu_k^*$ is equal to some $\frac{\pi^2 p^2}{c_k^2 k}+ \pi^2 q^2 c_k^2 k$, for $q \ge 1$. In both cases, there exists some $p$ such that one of the first $k+1$ eigenvalues is given by $(p,1)$. Let $\overline p$ be the smallest number such that
\begin{equation}\label{bbg02}
\frac{\pi^2 \overline p^2}{c_k^2 k}+ \pi^2 c_k^2 k \ge 4 \pi k -16(\pi k)^{1/2},
\end{equation}
and $(\overline p, 1)$ does not produce an eigenvalue of the list $\mu_0(R_{b_k^*}), \dots, \mu_k(R_{b_k^*})$.

Then all eigenvalues given by the pairs $(0,1), \dots, (\overline p -1, 1)$ belong to the list $\mu_0(R_{b_k^*}), \dots, \mu_k(R_{b_k^*})$. Now, we consider the eigenvalues given by the pairs
$$(0,0),(1,0),\dots, (k-\overline p +1, 0).$$
We conclude that the eigenvalue given by the last pair $(k-\overline p +1, 0)$ is not smaller than $\mu_k^*$. Consequently
\begin{equation}\label{bbg03}
\frac{\pi^2 (k-\overline p +1)^2}{c_k^2 k} \ge \mu_k^* \ge 4 \pi k -16(\pi k)^{1/2}.
\end{equation}
From \eqref{bbg02} and \eqref{bbg03}, we get, respectively
$$\pi \overline p \ge c_k (4\pi k^2 -16\pi^{1/2} k^{3/2} -\pi ^2 c_k^2 k^2)^{1/2}$$
$$\pi (k-\overline p +1) \ge c_k (4\pi k^2 -16\pi^{1/2} k^{3/2})^{1/2}. $$
Adding the two inequalities, dividing by $k$ and passing to the limit for ${k \rightarrow \infty}$, we obtain that, for any limit point $\alpha \in [0,(\pi/4)^{1/2}]$ of the sequence $(c_k)_k$,
$$\pi \ge \alpha ((4\pi)^{1/2} + (4\pi -\pi ^2 \alpha ^2)^{1/2}).$$
A numerical evaluation, gives that $\alpha \in [0, 0.46359]$.
\end{proof}

We now prove that $\limsup b_k^*<\infty$.
Since \eqref{e14} holds for all pairs $(\mu,b)$, it must also hold
for all optimal pairs $(\mu_k^*,b_k^*)$. Furthermore, we note that $\mu\mapsto N(\mu;b)$ is increasing. Then, $\mu_k^*$ being optimal and having finite multiplicity, we have for all $\epsilon\in (0,\nu_1)$ that
\begin{equation*}
k-1\ge N(\mu_k^*-\epsilon;b_k^*)\ge N(\nu_k-\epsilon;b_k^*).
\end{equation*}
By Lemmas \ref{lem2} and \ref{lem3}, we have that for all $\epsilon>0$ sufficiently small
\begin{equation*}
\limsup_{k\rightarrow\infty}\frac{(\nu_k-\epsilon)^{1/2}}{\pi b_k^*}\ge 2.
\end{equation*}
So invoking Lemma \ref{lem1}, for all $k$ sufficiently large, we obtain that
\begin{equation*}
k-1\ge N(\nu_k-\epsilon;b_k^*)\ge \frac{\nu_k-\epsilon}{4\pi} + \frac{b_k^*(\nu_k-\epsilon)^{1/2}}{2\pi} - \frac{(b_k^*)^{3/2}(\nu_k-\epsilon)^{1/4}}{(2\pi)^{1/2}} -1.
\end{equation*}
Rearranging terms we have that
\begin{equation}\label{e16}
\frac{4\pi k-\nu_k+\epsilon}{(\nu_k-\epsilon)^{1/2}}\ge 2b_k^*(1-(2\pi b_k^*)^{1/2}(\nu_k-\epsilon)^{-1/4}).
\end{equation}
By Lemma \ref{lem2}, we conclude that
\begin{equation}\label{e22}
\limsup_{k\rightarrow \infty} (\nu_k-\epsilon)^{-1/2}(4\pi k-\nu_k+\epsilon)\le 8.
\end{equation}
On the other hand, Lemma \ref{lem3} gives that
\begin{equation}\label{e26}
\liminf_{k\rightarrow \infty}(1-(2\pi b_k^*)^{1/2}(\nu_k-\epsilon)^{-1/4})\ge 1-\pi^{1/4}(0.46359)^{1/2}.
\end{equation}
Putting \eqref{e16}, \eqref{e22} and \eqref{e26} together gives that
\begin{equation}\label{e27}
\limsup_{k\rightarrow\infty}b_k^*\le \frac{4}{1-\pi^{1/4}(0.46359)^{1/2}}\le 43.
\end{equation}

\noindent{\it Proof of Theorem \ref{the1}(ii).}
Let
\begin{equation*}
N_0(\mu;b) = \bigg\vert\left\{ (x,y) \in \Z^{2} : \frac{\pi^{2}x^2}{b^{2}} + \pi^{2}b^{2}y^{2} \leq \mu \right\}\bigg\vert.
\end{equation*}
Then
\begin{equation}\label{e29}
N(\mu;b)=\frac14N_0(\mu;b)+\frac12\bigg\lfloor \frac{b\mu^{1/2}}{\pi}\bigg\rfloor+\frac12\bigg\lfloor \frac{\mu^{1/2}}{\pi b}\bigg\rfloor-\frac14.
\end{equation}
We apply the identity above to the optimal pair $(b_k^*,\mu_k^*)$, and obtain that if $\mu_k^*$ has multiplicity $\Theta_k$ then
\begin{align}\label{e30}
k+\Theta_k-1= N(\mu_k^*;b_k^*)&=\frac14N_0(\mu_k^*;b_k^*)+\frac12\bigg\lfloor \frac{{b_k^*}(\mu_k^*)^{1/2}}{\pi}\bigg\rfloor+\frac12\bigg\lfloor \frac{(\mu_k^*)^{1/2}}{b_k^*\pi}\bigg\rfloor-\frac14\nonumber \\ &\ge\frac14N_0(\mu_k^*;b_k^*)+ \frac{b_k^*(\mu_k^*)^{1/2}}{2\pi}+ \frac{(\mu_k^*)^{1/2}}{2\pi b_k^*}-\frac54.
\end{align}
By \eqref{e27}, we have that the $b_k^*$ are bounded uniformly in $k$. It is known by \cite {Hux} that there exist constants $C<\infty$ and, for any $\epsilon >0$ , $\frac12< \theta < \frac{131}{208} + \epsilon$ such that
\begin{equation}\label{e31}
\frac{\mu}{\pi}+C\mu^{\theta/2}+1\ge N_0(\mu;b)\ge \frac{\mu}{\pi}-C\mu^{\theta/2}.
\end{equation}
So by \eqref{e30} and \eqref{e31} we conclude that
\begin{equation*}
b_k^*+\frac{1}{b_k^*}\le \frac{4\pi k-\mu_k^*}{2(\mu_k^*)^{1/2}}+\frac{\pi C}{2}(\mu_k^*)^{(\theta-1)/2}+\frac{2\pi\Theta_k}{(\mu_k^*)^{1/2}}+\frac12,
\end{equation*}
where we have used that $\mu_k^*\ge \mu_1^*=\pi ^2$.
We observe that $\mu\mapsto \frac{4\pi k-\mu}{2\mu^{1/2}}+\frac{\pi C}{2\mu^{(1-\theta)/2}}+\frac{2\pi\Theta_k}{\mu^{1/2}}$ is decreasing. By the optimality of $\mu_k^*$, we have that
\begin{equation*}
b_k^*+\frac{1}{b_k^*}\le \frac{4\pi k-\nu_k}{2\nu_k^{1/2}}+\frac{\pi C}{2\nu_k^{(1-\theta)/2}}+\frac{2\pi\Theta_k}{\nu_k^{1/2}}+\frac12.
\end{equation*}
By \eqref{e29} and \eqref{e31}, we have that
\begin{equation}\label{e34}
k\le N(\nu_k;1)\le\frac{\nu_k}{4\pi}+\frac{\nu_k^{1/2}}{\pi}+\frac{C\nu_k^{\theta/2}}{4}.
\end{equation}
It follows that
\begin{equation*}
b_k^*+\frac{1}{b_k^*}\le 2+O(k^{(\theta-1)/2}),
\end{equation*}
and
\begin{equation*}
b_k^*=1+O(k^{(\theta-1)/4}).
\end{equation*}
This completes the proof of Theorem \ref{the1}(ii).
\qed

\noindent{\it Proof of Theorem \ref{the1}(iii).}
First, we obtain a lower bound for $\mu_k^*$. By its maximality we have that $\mu_k^*\ge \nu_k$, and so it suffices to obtain a lower bound for the latter.
By \eqref{e34}, we have that
\begin{equation*}
k\le \frac{\nu_k}{4\pi}+\left(\frac{4k}{\pi}\right)^{1/2}+ O(k^{\theta/2}),
\end{equation*}
where we have used P\'olya's Inequality $\nu_k\le 4\pi k$.
This proves the lower bound in \eqref{i3} since $(1+\theta)/4> \theta/2.$

To prove the upper bound we have by \eqref{e30}, \eqref{e31} and \eqref{i4} that
\begin{align*}
N(\mu_k^*;b_k^*)&\ge \frac{\mu_k^*}{4\pi}+\frac{b_k^*(\mu_k^*)^{1/2} }{2\pi}+ \frac{(\mu_k^*)^{1/2}}{2\pi b_k^*}-\frac{C}{4}(\mu_k^*)^{\theta/2}-\frac54\nonumber \\ &
\ge \frac{\mu_k^*}{4\pi}+\frac{b_k^*\nu_k^{1/2} }{2\pi}+ \frac{\nu_k^{1/2}}{2\pi b_k^*}-\frac{C}{4}(4\pi k)^{\theta/2}-\frac54\nonumber \\ &
=\frac{\mu_k^*}{4\pi}+\frac{\nu_k^{1/2}}{\pi}+ \nu_k^{1/2}O(k^{(\theta -1)/4})+O(k^{\theta/2}),
\end{align*}
where we have used the optimality of $\mu_k^*$ and P\'olya's Inequality: $\nu_k\le \mu_k^*\le 4\pi k$.
By Lemma \ref{lem2} and P\'olya's Inequality, we have that $\nu_k^{1/2}= (4\pi k)^{1/2}+O(1)$. This shows that, since $\theta<1$,
\begin{equation}\label{c2}
N(\mu_k^*;b_k^*)\ge \frac{\mu_k^*}{4\pi}+ \left(\frac{4k}{\pi}\right)^{1/2} +O(k^{(\theta+1)/4}).
\end{equation}
We note that the multiplicity $\Theta_k$ of $\mu_k^*$ is equal to the number of lattice points in the first quadrant lying on the curve
\begin{equation}\label{c3}
\frac{\pi^{2}x^{2}}{(b_k^*)^{2}} + \pi^{2}(b_k^*)^{2}y^{2}= \mu_k^*.
\end{equation}
The latter multiplicity is bounded by Theorem 1 in \cite{J}, and is of order $O(\ell^{2/3})$, where $\ell$ is the length of the curve defined in \eqref{c3},
which in turn equals $O((\mu_k^*)^{1/2})=O(k^{1/2})$. So the multiplicity of $\mu_k^*$ is bounded by $O(k^{1/3})$.
It follows by \eqref{c2} that
\begin{equation*}
O(k^{1/3})+k\ge \frac{\mu_k^*}{4\pi}+ \left(\frac{4k}{\pi}\right)^{1/2} +O(k^{(\theta+1)/4}).
\end{equation*}
This completes the proof of Theorem \ref{the1}(iii) since $\frac13<(1+\theta)/4$.

\section{Neumann eigenvalues with a perimeter constraint}\label{Per}

In general, the problems of maximising or minimising $\mu_k$ under a perimeter constraint are ill-posed. In fact, it is not
 difficult to see that for every $c >0$
 \begin{equation}\label{BBG}\inf\{\mu_k(\Omega) : \textup{$\Omega$ open, bounded with } \textup{Per} (\Omega)= c\}= 0,\end{equation}
\begin{equation}\label{bbg00}\sup\{\mu_k(\Omega) : \textup{$\Omega$ open, bounded with } \textup{Per} (\Omega)= c\}= +\infty.
\end{equation}
Indeed, the $k$th eigenvalue of a set $\Omega$ that is the disjoint union of $k+1$ balls is equal to $0$, so that the infimum under \eqref{BBG} is attained trivially. One can also construct a minimising sequence of connected sets where the $k$th eigenvalue tends to zero, for example, by connecting $k+1$ fixed disjoint balls with $k$ tubes of vanishing width (see \cite{Ar95}), while controlling the overall perimeter by rescaling.

For the maximisation problem, we construct the following example in $\R^2$. Let $\Lambda>0$ be arbitrary, and let $l>0$ be such that $l < \frac c4$, and $\frac{\pi^2}{l^2}\ge \Lambda$.  Let $\Omega$ be the square with vertices $(0,0), (l,0), (l,-l), (0, -l)$. Then $\mu_1(\Omega) = \frac{\pi^2}{l^2}$. Consider the function $\phi :\R\rightarrow \R$ defined by $\phi(x)= C\sin (\frac{2\pi}{l} x)$, where $C$ is such that  $\int_0^l \sqrt{1+(\phi'(x))^2} dx = c-3l$. We replace the edge between the first two vertices by the graph of the function $\frac 1n \phi (nx)$. In this way, we construct a set $\Omega_{n,l}$ with $\textup{Per}(\Omega_{n,l})=c.$ The sets $\Omega_{n,l}$ satisfy a uniform cone condition so that $\mu_1(\Omega_{n,l}) \rightarrow \mu_1(\Omega)=\frac{\pi^2}{l^2}$ as $n \rightarrow +\infty$. Hence for all $n$ sufficiently large $\mu_1(\Omega_{n,l})\ge\frac{\pi^2}{2l^2}\ge \frac{\Lambda}{2}$.
Since $\Lambda>0$ was arbitrary the supremum under \eqref{bbg00} is $+\infty$. The above example is easily extended to dimensions larger than $2$. We refer the reader to \cite{JMA} for related constructions.

Below we obtain some results for the variational problems \eqref{bbg04}, \eqref{bbg05}
with a perimeter constraint. We let $R_{a,b}$ denote a rectangle in $\R^2$ of side-lengths $a,b>0$
so that $\textup{Per} (R_{a,b})=2(a+b)$.

\subsection{Analysis of the maximisation problem \eqref{bbg04}.}\label{S3.1}
Our main theorem is the following.
\begin{theorem}\label{the2}
For $k \in \N$, there is a unique maximising rectangle $R_{a_k^*,b_k^*}$ with $a_k^* = \frac{2}{k+1} \in (0,1]$ and $b_k^*=2-a_k^*$
such that
\begin{equation*}
 \mu_k (R_{a_k^*,2-a_k^*}) =\frac{\pi^2 k^2}{(2-a_k^*)^2}= \frac{\pi^2}{(a_k^*)^2}=\frac{\pi^2(k+1)^2}{4},
\end{equation*}
i.e. $\mu_k^* = \mu_k (R_{a_k^*,2-a_k^*})$ is realised by the modes $(k,0)$ and $(0,1).$
\end{theorem}
\begin{proof}
We first show that for every $k \ge 0$, problem \eqref{bbg04} has a solution.

Fix $k \in \Z^+$ and let $(R_{a_n, 2-a_n})_n$, $a_n \in (0,1]$, be a maximising sequence of rectangles for $\mu_k$.
By taking a monotone subsequence if necessary, we may assume that $(a_n)_n$ converges. Let $a_k^*=\lim_{n\rightarrow\infty} a_n$.
Now, we claim that
\begin{equation}\label{bbg05a}
a_k^* \geq \frac{2}{k+1}.
\end{equation}

Suppose to the contrary that $a_k^* < \frac{2}{k+1}$. Then we have that
\begin{equation*}
\frac{\pi^2 k^2}{(2-a_k^*)^2} < \frac{\pi^2}{(a_k^*)^2},
\end{equation*}
where the right-hand side above is $+\infty$ in the case that $a_k^*=0$.
Hence, the $k$ eigenvalues that are given by the pairs
$(1,0),(2,0), \dots ,(k,0)$ are smaller than the eigenvalue
that is given by the pair $(0,1)$. So $\mu_k^* = \frac{\pi^2 k^2}{(2-a_k^*)^2}$.
However, if we consider $\tilde{a}_k \in (a_k^* ,\frac{2}{k+1})$, then
\begin{equation*}
\mu_{k}(\tilde{a}_k) = \frac{\pi^2 k^2}{(2-\tilde{a}_k)^2} > \frac{\pi^2 k^2}{(2-a_k^*)^2},
\end{equation*}
which contradicts the maximality of $\mu_k^*$. This proves \eqref{bbg05a}.

For $a_k = \frac{2}{k+1}$, we have that
\begin{equation*}
\mu_k(R_{a_k,2-a_k})= \frac{\pi^2 k^2}{(2-\frac{2}{k+1})^2} =\frac{\pi^2 (k+1)^2}{4}.
\end{equation*}
So, by maximality, we deduce that
\begin{equation}\label{bbg06a}
\mu_k (R_{a_k^*, 2-a_k^*}) \ge \frac{\pi^2 (k+1)^2}{4}.
\end{equation}

Let
\begin{equation*}
\mu_k (R_{a_k^*,2-a_k^*}) = \frac{\pi^2 p^2}{(2-a_k^*)^2} + \frac{\pi^2 q^2}{(a_k^*)^2},
\end{equation*}
for some $(p,q) \in (\Z^{+})^2$, $p+q \leq k$.

Below we show that $q\le 2.$ Suppose to the contrary that $q \geq 3$. Then, by P\'{o}lya's Inequality and since $a_k^* \in (0,1]$, we have that
\begin{equation*}
\frac{9\pi^2}{(a_k^*)^2} \leq \mu_k (R_{a_k^*,2-a_k^*}) \leq \frac{4\pi k}{a_k^*(2-a_k^*)} \leq \frac{4\pi k}{a_k^*},
\end{equation*}
which implies that
\begin{equation*}
a_k^* \geq \frac{9\pi}{4k}.
\end{equation*}
Hence, we have that
\begin{equation*}
\mu_k^* \leq \frac{4\pi k}{a_k^*} \leq \frac{16k^2}{9} < \frac{\pi^2 (k+1)^2}{4}.
\end{equation*}
This contradicts \eqref{bbg06a}. So, for all $k \in \Z^+$, $\mu_k^*$ has $q \leq 2$.

Now we consider the case where $q=2$, and note that
\begin{equation*}
\frac{4\pi^2}{(a_k^*)^2} > \frac{\pi^2}{(2-a_k^*)^2} + \frac{\pi^2}{(a_k^*)^2},
\end{equation*}
since $a_k^* \in (0, 1]$. This shows that the eigenvalues given by the pairs $(0,1)$ and $(1,1)$ are strictly smaller than the one given by the pair $(0,2)$.
Below we will show that the eigenvalues given by the pairs  $(0,0),(1,0), \dots, (k-2,0)$ are also strictly smaller than the eigenvalue given by the pair $(0,2)$.
By \eqref{bbg06a} and by P\'{o}lya's Inequality, we have that
\begin{equation*}
\frac{\pi^2 (k+1)^2}{4} \leq \mu_k (R_{a_k^*,2-a_k^*}) \leq \frac{4\pi k}{a_k^*(2-a_k^*)},
\end{equation*}
which implies that
\begin{align}
a_k^*(2-a_k^*) &\leq \frac{16}{\pi (k+1)} \label{bbg13}
\end{align}
Since $a_k^*(2-a_k^*)\le 1$, we see that \eqref{bbg13} does not give any information about $a_k^*$ for $k=1,2,3,4.$
We first consider the case $k\ge 5$.
By solving \eqref{bbg13}, and taking into account that $a_k^*\le 1,$ we have that
\begin{equation}\label{bbg17}
a_k^* \leq 1- \sqrt{1-(16/\pi (k+1))}.
\end{equation}
We wish to show that
\begin{equation}\label{bbg17a}
\frac{4\pi^2}{(a_k^*)^2}> \frac{\pi^2(k-2)^2}{(2-a_k^*)^2}.
\end{equation}
This is equivalent to showing that $a_k^*<\frac4k$. The latter is clearly satisfied if $1- \sqrt{1-(16/\pi (k+1))}<\frac4k$. After elementary arithmetic, we see that this is equivalent to
\begin{equation}\label{bbg17b}
k>\frac{\pi}{\pi-2}+\frac{2\pi}{(\pi-2)k}.
\end{equation}
Since $k\ge 5,$ we have that the right-hand side of \eqref{bbg17b} is bounded from above by $\frac{7\pi}{5(\pi-2)}<5$. So \eqref{bbg17a} holds for $k\ge 5$.
So the eigenvalues that are given by the pairs $(0,0),(1,0), \dots, (k-2,0),(0,1),(1,1)$
are all strictly smaller than the one that is given by the pair $(0,2)$, and there are $k+1$ of them.
Hence $\mu_k^*$ cannot have $q=2$. Thus $q=0$ or $q=1$.

Either $q=0$ and $p=k$, $\mu_k^*=\frac{\pi^2 k^2}{(2-a_k^*)^2}$ and the first $k+1$ eigenvalues
are given by the pairs $(0,0), (1,0), \dots, (k,0)$. In this case, $\mu_k^*$ cannot
be simple. Otherwise, the derivative of the mapping $a \mapsto \mu_k(R_{a,2-a})$ with respect to $a$ would be
non-vanishing as before, thus contradicting the maximality of $\mu_k^*$. Hence $\mu_k^* = \frac{\pi^2}{(a_k^*)^2}$,
i.e. $\mu_k^*$ is realised by the modes $(k,0)$ and $(0,1)$.

Or one of the first $k+1$ eigenvalues is given by a pair $(p,1), p \in \Z^+$.
Let $\bar {p}$ be the smallest number such that
\begin{equation*}
\frac{\pi^2 \bar{p}^2}{(2-a_k^*)^2} + \frac{\pi^2}{(a_k^*)^2} > \mu_k^*.
\end{equation*}
Then all eigenvalues given by the pairs $(0,1), (1,1), \dots , (\bar{p}-1,1)$ are in the list
$$\mu_{0}(R_{a_k^*,2-a_k^*}), \mu_{1}(R_{a_k^*,2-a_k^*}), \dots , \mu_{k}(R_{a_k^*,2-a_k^*}).$$
By considering the eigenvalues given by the pairs $(0,0), (1,0), \dots, (k-\bar{p}+1,0)$,
we deduce that
\begin{equation*}
\frac{\pi^2(k-\bar{p}+1)^2}{(2-a_k^*)^2} \geq \mu_k^* \geq \frac{\pi^2 (k+1)^2}{4}.
\end{equation*}
Thus we have that
\begin{equation*}
\bar{p}\le \frac12(k+1)a_k^*,
\end{equation*}
which, together with \eqref{bbg17}, gives that
\begin{align}\label{bbg23}
\bar{p}&\le \frac12(k+1)(1- \sqrt{1-(16/\pi (k+1))})\nonumber \\ &=
\frac{8}{\pi}\bigg(1+\bigg(1-\frac{16}{\pi(k+1)}\bigg)^{1/2}\bigg)^{-1}.
\end{align}
The right-hand side of \eqref{bbg23} is decreasing in $k$. So for $k\ge 5$ we have that the right-hand side of \eqref{bbg23}
is bounded from above by $\frac{8}{\pi}\bigg(1+\bigg(1-\frac{8}{3\pi}\bigg)^{1/2}\bigg)^{-1}<2.$
Hence $\bar{p}=1$, since $\bar{p} \in \N$. Therefore $p=0$.
So $p=0$ and the first $k+2$ eigenvalues are given by the pairs $(0,0), (1,0), \dots, (k,0), (0,1)$.
Then, as before, $\mu_k^*=\frac{\pi^2 k^2}{(2-a_k^*)^2} = \frac{\pi^2}{(a_k^*)^2}$, since in either case
$\mu_k^*$ cannot be simple, i.e. $\mu_k^*$ is realised by the modes $(k,0)$ and $(0,1)$.

It remains to deal with the cases $k=1,2,3,4$.

Let $a_1 \in (0,1]$. Then
\begin{equation*}
\mu_1(R_{a_1,2-a_1}) = \frac{\pi^2 p^2}{(2-a_1)^2}+\frac{\pi^2 q^2}{a_1^2}
\end{equation*}
for either the pair $(1,0)$ or the pair $(0,1)$. Since $a_1 \in (0,1]$,
$\mu_1(R_{a_1,2-a_1})=\frac{\pi^2}{(2-a_1)^2}$. This is maximal for $a_1=1$.
Hence $\mu_1^* = \pi^2$ with $a_1^* = 1$ and corresponding modes $(1,0),(0,1)$.

Let $a_2 \in (0,1]$. Then
\begin{equation*}
\mu_2(R_{a_2,2-a_2}) = \frac{\pi^2 p^2}{(2-a_2)^2}+\frac{\pi^2 q^2}{a_2^2},
\end{equation*}
with $p\leq 2, q \leq 2$ and $p+q \leq 2$.
The possible pairs that give $\mu_2(R_{a_2,2-a_2})$ are
$$(2,0),(1,0),(1,1),(0,1),(0,2).$$
Now $\mu_1(R_{a_1,2-a_1}) = \frac{\pi^2}{(2-a_1)^2}$ is given by the pair $(1,0)$.
So $\mu_2(R_{a_2,2-a_2})$ must be given by either $(2,0)$ or $(0,1)$.
We have that
\begin{equation*}
\frac{4\pi^2}{(2-a_2)^2} \leq \frac{\pi^2}{a_2^2} \iff a_2 \leq \frac23,
\end{equation*}
hence
\begin{equation*}
\mu_2(R_{a_2,2-a_2}) =
\begin{cases}
\frac{4\pi^2}{(2-a_2)^2}, & 0<a_2\leq \frac23,\\
\frac{\pi^2}{a_2^2}, & \frac23 \leq a_2 \leq 1.
\end{cases}
\end{equation*}
Thus we obtain that $\mu_2^* = \frac{9\pi^2}{4}$ with $a_2^* = \frac23$ and
corresponding modes $(2,0),(0,1)$.

Let $a_3 \in (0,1]$. Then
\begin{equation*}
\mu_3(R_{a_3,2-a_3}) = \frac{\pi^2 p^2}{(2-a_3)^2}+\frac{\pi^2 q^2}{a_3^2},
\end{equation*}
with $p\leq 3, q \leq 2$ and $p+q\leq 3$.
The possible pairs that give $\mu_3(R_{a_3,2-a_3})$ are
$$(3,0),(2,0),(1,0),(2,1),(1,1),(0,1),(1,2),(0,2).$$

For $0<a_2\leq \frac23$, $\mu_2(R_{a_2,2-a_2}) = \frac{4\pi^2}{(2-a_2)^2}$ is given by the pair $(2,0)$.
So for $0<a_3\leq \frac23$, $\mu_3(R_{a_3,2-a_3})$ must be given by either $(3,0)$ or $(0,1)$.
We have that
\begin{equation*}
\frac{9\pi^2}{(2-a_3)^2} \leq \frac{\pi^2}{a_3^2} \iff a_3\leq \frac12.
\end{equation*}
In addition, for $\frac23 \leq a_2 \leq 1$, $\mu_2(R_{a_2,2-a_2}) = \frac{\pi^2}{a_2^2}$ is given by
the pair $(0,1)$. So for $\frac23 \leq a_3 \leq 1$, $\mu_3(R_{a_3,2-a_3})$ must be given by
either $(2,0)$ or $(1,1)$. We have that
\begin{equation*}
\frac{4\pi^2}{(2-a_3)^2} \leq \frac{\pi^2}{(2-a_3)^2} + \frac{\pi^2}{a_3^2} \iff a_3 \leq \sqrt{3}-1.
\end{equation*}
Thus, we obtain that
\begin{equation*}
\mu_3(R_{a_3,2-a_3}) =
\begin{cases}
\frac{9\pi^2}{(2-a_3)^2}, &0<a_3\leq \frac12,\\
\frac{\pi^2}{a_3^2}, &\frac12 \leq a_3 \leq \frac23,\\
\frac{4\pi^2}{(2-a_3)^2}, &\frac23 \leq a_3 \leq \sqrt{3}-1,\\
\frac{\pi^2}{(2-a_3)^2} + \frac{\pi^2}{a_3^2}, & \sqrt{3}-1 \leq a_3 \leq 1.
\end{cases}
\end{equation*}
We deduce that $\mu_3^* = 4\pi^2$ with $a_3^* = \frac12$ and corresponding modes $(3,0),(0,1)$.

Let $a_4 \in (0,1]$. Then
\begin{equation*}
\mu_4(R_{a_4,2-a_4}) = \frac{\pi^2 p^2}{(2-a_4)^2}+\frac{\pi^2 q^2}{a_4^2},
\end{equation*}
with $p\leq 4, q \leq 2$ and $p+q\leq 4$.
The possible pairs that give $\mu_4(R_{a_4,2-a_4})$ are
$$(4,0),(3,0),(2,0),(1,0),(3,1),(2,1),(1,1),(0,1),(2,2),(1,2),(0,2).$$

For $0<a_3\leq \frac12$, $\mu_3(R_{a_3,2-a_3}) = \frac{9\pi^2}{(2-a_3)^2}$
is given by the pair $(3,0)$. So for $0<a_4\leq \frac12$, $\mu_4(R_{a_4,2-a_4})$
must be given by either $(4,0)$ or $(0,1)$.
We have that
\begin{equation*}
\frac{16\pi^{2}}{(2-a_4)^2} \leq \frac{\pi^2}{a_4^2} \iff a_4\leq \frac25.
\end{equation*}
In addition, for $\frac12 \leq a_2,a_3 \leq \frac23$, $\mu_3(R_{a_3,2-a_3}) = \frac{\pi^2}{a_3^2}$
is given by the pair $(0,1)$, and $\mu_2(R_{a_2,2-a_2}) = \frac{4\pi^2}{(2-a_2)^2}$ is given by the
pair $(2,0)$. So for $\frac12 \leq a_4 \leq \frac23$, $\mu_4(R_{a_4,2-a_4})$
must be given by either $(3,0), (1,1)$ or $(0,2)$. We have that
\begin{align*}
\frac{9\pi^2}{(2-a_4)^2} \leq \frac{\pi^2}{(2-a_4)^2}+\frac{\pi^2}{a_4^2}
&\iff a_4 \leq \frac27 (\sqrt{8}-1),\\
\frac{9\pi^2}{(2-a_4)^2} \leq \frac{4\pi^2}{a_4^2}
&\iff a_4 \leq \frac45,\\
\frac{\pi^2}{(2-a_4)^2}+\frac{\pi^2}{a_4^2} \leq \frac{4\pi^2}{a_4^2}
&\Leftarrow \quad a_4 \in (0,1].
\end{align*}
For $\frac23 \leq a_3 \leq \sqrt{3}-1$, $\mu_3(R_{a_3,2-a_3}) = \frac{4\pi^2}{(2-a_3)^2}$
is given by the pair $(2,0)$. Similarly to the above, for $\frac23 \leq a_4 \leq \sqrt{3}-1$,
$\mu_4(R_{a_4,2-a_4})$ must be given by either $(3,0)$ or $(1,1)$.

Finally, for $\sqrt{3}-1 \leq a_3 \leq 1$, $\mu_3(R_{a_3,2-a_3}) = \frac{\pi^2}{(2-a_3)^2} + \frac{\pi^2}{a_3^2}$
is given by the pair $(1,1)$. So for $\sqrt{3}-1 \leq a_4 \leq 1$, $\mu_4(R_{a_4,2-a_4})$ must be given by $(2,0)$,
as $(1,0),(0,1),(1,1)$ have already been used for this range of $a$ by $\mu_1, \mu_2, \mu_3$ respectively.

Hence, we obtain that
\begin{equation*}
\mu_4(R_{a_4,2-a_4})=
\begin{cases}
\frac{16\pi^2}{(2-a_4)^2}, &0<a_4\leq \frac25,\\
\frac{\pi^2}{a_4^2}, &\frac25 \leq a_4 \leq \frac12,\\
\frac{9\pi^2}{(2-a_4)^2}, &\frac12 \leq a_4 \leq \frac27 (\sqrt{8}-1),\\
\frac{\pi^2}{(2-a_4)^2} + \frac{\pi^2}{a_4^2}, & \frac27(\sqrt{8}-1) \leq a_4 \leq \sqrt{3}-1,\\
\frac{4\pi^2}{(2-a_4)^2} &\sqrt{3}-1 \leq a_4 \leq 1.
\end{cases}
\end{equation*}
Thus $\mu_4^* = \frac{25\pi^2}{4}$ with $a_4^* = \frac25$ and corresponding modes $(4,0),(0,1)$.
\end{proof}

\subsection{Analysis of the minimisation problem \eqref{bbg05}.}\label{S3.2}

Our main result is the following.

\begin{theorem}\label{the3}
\begin{enumerate}
\item[\textup{(i)}] If $k=1$, then variational problem \eqref{bbg05} does not have a minimiser, and the infimum equals $\frac{\pi^2}{4}$.
\item[\textup{(ii)}] If $k\ge 2$, then variational problem \eqref{bbg05} does have a minimiser.
\item[\textup{(iii)}] If $k\ge 2$ and $R_{a_k^*,b_k^*},{ \, a_k^*\in(0,1],}\,b_k^* = 2-a_k^* $ are minimisers, then
\begin{equation*}
\lim_{k \rightarrow \infty} a_k^* =1,
\end{equation*}
i.e. any sequence of optimal rectangles for Problem \eqref{bbg05} converges to the unit square, as ${k \rightarrow \infty}$.
\end{enumerate}
\end{theorem}
\begin{proof}
If $k=1$, then $(R_{\frac{1}{n}, 2-\frac{1}{n}})_n$ is minimising and collapses to a segment of length $2$. This proves the assertion under (i).

To prove (ii), we fix $k \ge 2,$ and consider a minimising sequence
for problem  \eqref{bbg05}, $(R_{a_n, 2-a_n})_n$ with $a_n \in (0,1].$
By taking a monotone subsequence if necessary, $(a_n)_n$ converges. Then $(a_n)_n$ cannot
converge to 0. If $a_n \rightarrow 0$, then for $n$ large enough such that
$0< a_n \leq \frac{2}{k+1}$, we have that
\begin{equation*}
\mu_k (R_{a_n, 2-a_n})= \frac{\pi^2 k^2}{(2-a_n)^2} \rightarrow  \frac{\pi^2 k^2}{4}.
\end{equation*}
However, by minimality and by P\'olya's Inequality, we have that
\begin{equation*}
\mu_k (R_{a_n, 2-a_n}) \le \nu_k\le4\pi k.
\end{equation*}
Clearly, this inequality leads to a contradiction as soon as
$\frac{\pi^2 k^2}{4} > 4\pi k$. That is the case for $k \ge 6$.
So, for $k \geq 6$, $a_n \rightarrow a_k^* >0$, which gives an
optimal rectangle, $R_{a_k^*,2-a_k^*}$.

Similarly to Subsection (3.1), we obtain the values of $\mu_k^* (R_{a_k^*,2-a_k^*})$ for $k=2,3,4,5$ by direct computation.
In the table below, we list these values as well as the corresponding values of $a_k^*$ and the minimising modes.

\begin{center}
 \begin{tabular}{|c|c|c|c|}
  \hline
  $k$ & $\mu_k^*$ & $a_k^*$ & \textup{Minimising modes}\\
  \hline
  2 & $\pi^2$ & $1$ & (1,0),(0,1) \\
  \hline
  3 & $2\pi^2$ & $1$ & (1,1) \\
  \hline
  4 & $ \displaystyle \frac 23\pi^2(2+\sqrt{3})$ & $\sqrt{3}-1$ & (2,0),(1,1)\\
  \hline
  5 & $4\pi^2$ & {$1$} & {(2,0),(0,2)}\\
  \hline
\end{tabular}
\end{center}
We note that a degenerating sequence of rectangles $R_{a_2^{(n)},2-a_2^{(n)}}$ with $a_2^{(n)} \rightarrow 0$,
gives $\mu_2(a_2^{(n)}) \rightarrow \pi^2$.
In addition, we remark that $\mu_3^*$ has only one minimising mode $(1,1)$. By considering the derivative of the function
$\frac{\pi^2}{(2-a)^2} + \frac{\pi^2}{a^2}$ with respect to $a$, we see that the point $a=1$ is a minimum point.
This is due to the fact that for the mode $(1,1)$ it is possible to obtain a vanishing derivative.

To prove assertion (iii) of the theorem we note that
by minimality and P\'{o}lya's Inequality,
\begin{equation*}
\mu_k (R_{a_k^*, 2-a_k^*}) \le \nu_k\le 4\pi k.
\end{equation*}
Recall that if $R_{a_1, b_1}, R_{a_2, b_2}$ are two rectangles such that $a_1\le a_2$ and $b_1\le b_2$, then for every $k \ge 0$
$\mu_k(R_{a_1, b_1}) \ge \mu_k(R_{a_2, b_2}).$ The latter is a direct consequence of the expression of the eigenvalues on rectangles.
Assume for some subsequence (still denoted with the same index $k$) that $a_k^* \rightarrow \alpha$.
Then, for every $\delta >0$, there exists $K_\delta$ such that for $k \ge K_\delta$ we have that
\begin{equation*}
R_{a_k^*, 2-a_k^*} \subset R_{\alpha + \delta, 2-\alpha + \delta}.
\end{equation*}
We have that
\begin{equation*}
\mu_k (R_{\alpha + \delta, 2-\alpha + \delta}) \le \mu_k (R_{a_k^*, 2-a_k^*}) \le \nu_k\le 4\pi k.
\end{equation*}
Using the Weyl asymptotic on $R_{\alpha + \delta, 2-\alpha + \delta}$, and letting
${k \rightarrow \infty}$, we obtain
\begin{equation*}
\frac{4\pi}{ (\alpha + \delta)(2-\alpha + \delta)} \le 4\pi.
\end{equation*}
By subsequently letting $\delta \rightarrow 0$, we obtain that
$\alpha (2-\alpha) \ge 1$, which leads to $\alpha=1$.
Hence, $\lim_{k \rightarrow \infty} a_k^* =1$, and this limit is independent of the subsequence $(a_k^*)$.
\end{proof}

 It was shown in \cite{AF2} that the corresponding sequence of minimisers for Dirichlet eigenvalues on rectangles with a perimeter constraint
converges to the square with perimeter 4 as $k \rightarrow \infty$. A similar result holds in higher dimensions, and estimates for the rate of Hausdorff convergence were obtained (\cite{AF2}).

\subsection{Further remarks on higher dimensions.}\label{S3.3}

We conclude with some remarks on the higher-dimensional analogues of the problems
that we investigated in this paper.

If $m\ge 3$ then problem \eqref{bbg05} with
fixed $k$ does not have a solution, since a sequence of
cuboids with one very long edge has vanishing $k$th eigenvalue.

In order to analyse problem \eqref{bbg04}, we first observe that for every $k\ge 1$
and every $m \ge 2$ the problem
\begin{equation*}
\max\{\mu_k (R) : R \mbox{  cuboid}, R \subseteq {\mathbb R}^m, \;|R|=1\},
\end{equation*}
has a solution. Indeed, if a maximising sequence is degenerating, then one of the edges
of the cuboid is vanishing and so another one is blowing up. This second phenomenon
produces vanishing eigenvalues, so it is excluded.

Now, concerning problem  \eqref{bbg04} in $\R^m$, $ m \geq 3$,
we claim that there exists a solution. Indeed, a maximising sequence of cuboids
cannot have two (or more) vanishing edges, since this implies that another edge
is blowing up, so the $k$th eigenvalue is vanishing. There are only two possibilities: either
there is convergence to a non-degenerate cuboid, or (only) one edge is vanishing.
In the latter case, for a sufficiently short edge, the eigenvalues of
the cuboid will be given by the eigenvalues of the $(m-1)$-dimensional complement
cuboid that satisfies a volume constraint.
That is, if $\left(R_{a_1^{(n)}, \dots, a_m^{(n)}}\right)_n$ is a maximising sequence of
cuboids such that for all $i \in \{1, \dots , m\}$, $a_i^{(n)} \rightarrow a_i$
and, without loss of generality, $a_1^{(n)} \rightarrow 0$, then the perimeter constraint
becomes $a_2 a_3 \dots a_{m} = 4$. Thus, the eigenvalues of $R_{a_1, \dots ,a_m}$ are
the eigenvalues of the $(m-1)$-dimensional cuboid with edges of length $a_2, a_3, \dots,
a_m$ subject to a volume constraint. At this point, making the vanishing edge longer
would increase the eigenvalues.

For every $k$, let $R_k^*$ be a maximising cuboid.
Then, for $k \rightarrow \infty$ the sequence $(R_k^*)_k$
has to collapse. By considering the Weyl asymptotic on $R_k^*$,
$\mu_k^*$ would behave like $k^\frac 2m$. However, if one chooses a particular
sequence that collapses towards a fixed $(m-1)$-dimensional cuboid, then
$\mu_k^*$ would behave like $k^\frac {2}{m-1}$.

\bigskip


\end{document}